\documentclass[11pt,reqno]{amsart}

\usepackage[english]{babel}
\usepackage{amsthm}
\usepackage{amsmath}
\usepackage{amssymb}
\usepackage{amstext}
\usepackage{latexsym}
\usepackage{enumitem}
\usepackage{mathrsfs}
\usepackage{stmaryrd}
\usepackage{comment}
\usepackage{hyperref}
\usepackage{url}
\usepackage[left=1.9cm, right=1.9cm,marginparwidth=3cm, marginparsep=0.35cm]{geometry}

\usepackage{tikz} 
\usetikzlibrary{patterns}
\usepackage{tikz-3dplot}

\theoremstyle{plain}
\newtheorem{theorem}{Theorem}
\newtheorem{lemma}[theorem]{Lemma}

\newtheorem{corollary}[theorem]{Corollary}
\newtheorem{conjecture}{Conjecture}

\providecommand{\Bcal}{\mathcal{B}}

\providecommand{\DD}{\mathsf{D}}

\providecommand{\NNb}{\mathbb{N}}

\providecommand{\ZZ}{\mathbb{Z}}

\providecommand\llb{\llbracket}
\providecommand\rrb{\rrbracket}

\begin{document}
\title{The Davenport constant of low-dimensional boxes}

\author{Benjamin Girard}
\address{Sorbonne Université, Université Paris Cité, CNRS, IMJ-PRG, \'Equipe Combinatoire et Optimisation, F-75005 Paris, France}
\email{benjamin.girard@imj-prg.fr}

\author{Alain Plagne}
\address{Centre de math\'ematiques Laurent Schwartz, CNRS, \'Ecole polytechnique, Institut polytechnique de Paris, 91128 Palaiseau cedex, France}
\email{alain.plagne@polytechnique.edu}

\begin{abstract}
In this paper, we study the maximal length of a minimal zero-sum sequence over $\llbracket -m,m \rrbracket^d$ having a support of size $d+1$. 
When $d=2$, we find the exact value of this maximal length for all $m \ge 2$, as well as the exact structure of all extremal sequences for this problem. 
When $d=3$, we obtain lower and upper bounds that asymptotically match each other as $m$ tends to infinity.  
This allows us to improve on the best known lower bounds for the Davenport constant of low-dimensional boxes.

\end{abstract}

\subjclass[2020]{Primary: 11B75; Secondary: 11B30, 11P70} 
\keywords{Additive combinatorics, Davenport constant, inverse theorem, minimal zero-sum sequence} 

\maketitle

\section{Introduction}

Given any two integers $a \leq b$, the set of all integers $i$ satisfying $a \leq i \leq b$ will be denoted by $\llbracket a,b \rrbracket$. 
Whenever $d$, $m_1,\dots,m_d$ and $M_1,\dots,M_d$ are positive integers, the cartesian product $\prod^d_{i=1} \llbracket -m_i,M_i \rrbracket$ will often be informally referred to as a \em box \em of $\mathbb{Z}^d$.

Let $(G,+)$ be an abelian group and $X$ be a subset of $G$. 
A finite sequence $S = x_1, \dots, x_{\ell}$ of $\ell$ elements from $X$
will be called a \em sequence over \em $X$ of \em length \em $\|S\|=\ell$. The set $\text{supp}(S)=\{x_i : i \in \llbracket 1,\ell \rrbracket\}$ of all elements appearing at least once in the sequence $S$ will be called its \em support\em. 

The \em sum \em of $S$ is defined as $\sigma(S)=\sum_{i=1}^{\ell} x_i$. Whenever $\sigma(S)=0$, $S$ will be called a {\em zero-sum sequence}. Moreover, if $S$ is a zero-sum sequence such that $\sum_{i \in I} x_i \ne 0$ for every $\emptyset \neq I \subsetneq \llbracket 1, \ell \rrbracket$, it will be called a \em minimal \em zero-sum sequence. 

The study of minimal zero-sum sequences, known as {\em zero-sum theory}, is a well-established research topic in additive combinatorics with a wide range of applications, most notably to non-unique factorization theory \cite{GB,Ge-HK06a,Grynkie22}. 

A fundamental problem in this area asks for the \em Davenport constant \em of $X$, which is denoted by $\mathsf{D}(X)$ and defined as the maximal length of a minimal zero-sum sequence over $X$.
This question has been extensively studied since the late sixties, but remains widely open to this day, even in the special case where $X=G$ is a finite abelian group \cite{GG06,G18}, or where $X$ is a box of $G=\mathbb{Z}^d$ \cite{DGS,PT}.
In the latter case, here is a brief account of what is currently known.

When $d=1$, it is an easy exercise to prove that $\DD (\llbracket -1,1 \rrbracket )=2$ and $\DD (\llbracket -m,m \rrbracket )=2m-1$ for every integer $m \ge 2$, but it wasn't until recently that the authors of the present paper proved \cite{GP25} that, for any two positive integers $m,M$, one has
$$
\DD (\llbracket -m,M \rrbracket )= m+M - \rho (m,M)
$$
where 
$$
\rho (m,M) = \min \{ t \in \NNb : \text{ there is a } t' \in \NNb \text{ such that } 0 \leq t' \leq t \text{ and } \gcd \big( M-t', m- (t-t') \big)=1\}.
$$
Moreover, the exact structure of all minimal zero-sum sequences over $\llbracket -m,M \rrbracket$ having maximal length is known when $\rho(m,M)=0$ \cite{PT} or $\rho(m,M)=1$ \cite{GP26}.

When $d =2$, the Davenport constant of $\llbracket -1,1 \rrbracket \times \llbracket -m,M \rrbracket$  and of $\llbracket -1,m \rrbracket \times \llbracket -1,M \rrbracket$, as well as the exact structure of all minimal zero-sum sequences of maximal length, have been determined in \cite{China} and \cite{DZ21} respectively. 
It is also known \cite{China,GP26,PT} that $ \DD(\llb -1,1\rrb^2)=4$ and that, for every $m \ge 2$, one has
\begin{equation}
\label{DDbox}
4m^2-m+1 \le \DD(\llb -m,m\rrb^2) \le \frac{5 \pi}{2} m^2 + O(m).
\end{equation}
Finally, when $d\ge 3$, the best currently known bounds for $m \ge 2$ \cite{PT} are
\begin{equation}
\label{DDbox3et plus}
(2m-1)^d \le \DD(\llb -m,m\rrb^d) \le \left( 2 \left(d+\frac1d -1 \right) m+1 \right)^d.
\end{equation}

\section{New results}
\label{newresults}

Our aim in this paper is to improve on the lower bounds in (\ref{DDbox}) and (\ref{DDbox3et plus}) when $d \in \{2,3\}$. To do so, we study the quantity $\DD^{(d+1)} (\llbracket -m,m \rrbracket ^d)$ which is defined as the maximal length of a minimal zero-sum sequence over $\llbracket -m,m \rrbracket ^d$ having a support of size $d+1$.

These numbers $\DD^{(k)}(X)$, where $X$ is any subset of an abelian group $G$, and $k$ any positive integer, were introduced in \cite{GP26}, and studied therein in the special case where $k=d+1$ and $X$ is a discrete Euclidean ball of $\mathbb{Z}^d$, that is to say a subset of the form
$$
X= \Bcal^{(d)}_m =\{ (u_1,\dots, u_d) \in \ZZ^d : u_1^2+ \cdots +u_d^2 \leq m^2 \},
$$
where $m$ is a positive integer. 
Generally speaking, it is believed that whenever $X$ is a ``sufficiently regular" subset of $\mathbb{Z}^d$ (for example, a discrete ball for some norm over $\mathbb{R}^d$ with a radius being large enough in terms of $d$), the quantity $\DD^{(d+1)}(X)$ ought to be a good approximation, if not the exact value, of $\DD(X)$. For more insight, see the discussion above Theorem $2$ in \cite{GP26}.

Our first theorem bears upon two-dimensional boxes and will be proved in Section \ref{sectThm4}. It gives the exact value of $\DD^{(3)} (\llbracket -m,m \rrbracket ^2)$ (thus improving on \eqref{DDbox}, including the version in \cite{China}) for all integers $m \ge 2$, and also provides the exact structure of all minimal zero-sum sequences over $\llbracket -m,m \rrbracket ^2$ having a support of size $3$ and maximal length with respect to this property. 

Before stating our result, let us denote by $q(m)$ the smallest prime number not dividing $m$, and by $S(m)$ the sequence consisting of: 

$q(m)(m -1) + m$ copies of $
\begin{pmatrix}
m \\
m 
\end{pmatrix},
$
$2m^2 - mq(m)$ copies of 
$
\begin{pmatrix}
-m \\
m-1
\end{pmatrix},
$ 
and
$2m^2-m$ copies of 
$
 \begin{pmatrix}
 m-q(m) \\
-m
\end{pmatrix}
$.

\begin{theorem}
\label{th1}
For every integer $m \geq 2$, one has
$$
\DD^{(3)} (\llbracket -m,m \rrbracket ^2) = 4m^2 -q(m).
$$
Moreover, up to symmetries and reordering of terms, $S(m)$ is the only minimal zero-sum sequence 
over $\llbracket -m,m \rrbracket ^2$ having a support of size $3$ and maximal length with respect to this property.
\end{theorem}

This theorem readily allows us to improve on the lower bound given by \eqref{DDbox} when $d=2$. 

\begin{corollary}
\label{improvement44}
For every integer $m \geq 2$, one has 
$$
\DD (\llbracket -m,m \rrbracket ^2) \geq 4m^2 -q(m).
$$
In particular, 
$$
\DD (\llbracket -m,m \rrbracket ^2) \geq 4m^2 - O( \log m).
$$
\end{corollary}

The second assertion will be deduced from the first via a result on the asymptotic behaviour of the arithmetic function $q$ (see  Lemma \ref{qq} in Section \ref{sectThm4}).

More importantly, we conjecture that the first inequality in Corollary \ref{improvement44} is in fact an equality, and that all minimal zero-sum sequences of maximal length over $\llbracket -m,m \rrbracket ^2$ have a support of size $3$.
Recall that $S(m)$ is the sequence described just before the statement of Theorem \ref{th1}.

\begin{conjecture}
\label{conj8999}
For every integer $m \geq 2$, one has $$\DD (\llbracket -m,m \rrbracket ^2) =4m^2-q(m).$$ 
Moreover, up to symmetries and reordering of terms, $S(m)$ is the only minimal zero-sum sequence of maximal length over $\llbracket -m,m \rrbracket ^2$.
\end{conjecture}

In the three-dimensional case, we checked first that 
$\DD^{(4)} ( \llbracket -1,1 \rrbracket ^3)= 10$ (an exhaustive search can be done easily in this case). For all larger values of $m$, we obtain the following theorem.

\begin{theorem}
\label{thmbox3}
For every integer $m \geq 2$, one has
$$
\DD^{(4)} ( \llbracket -m,m \rrbracket ^3) \geq 
\left\{
\begin{array}{ll}
16m^3 - 16m^2 +10m-2 	& {\mathrm{ if }}\ m\  \mathrm{ is\ odd,} \\
16m^3 - 16m^2 +8m-1 	& \mathrm{ if }\ m\  \mathrm{ is\ even.}
\end{array}
\right.
$$
\end{theorem}

As a direct corollary, Theorem \ref{thmbox3} yields a new lower bound for the Davenport constant of $\llbracket -m,m \rrbracket ^3$, turning out to be twice as good as the one given by \eqref{DDbox3et plus} when $d=3$.

\begin{corollary}
\label{thmbox3456}
For every integer $m \geq 2$, one has
$$
\DD ( \llbracket -m,m \rrbracket ^3) \gtrsim 16m^3.
$$
\end{corollary}

Theorem \ref{thmbox3}, which will be proved in Section \ref{d=3}, follows from the explicit construction of minimal zero-sum sequences over $\llbracket -m,m \rrbracket ^3$. As in the two-dimensional case, we conjecture that these constructions are the longest possible ones (i.e. their length is equal to $\DD( \llbracket -m,m \rrbracket ^3)$), and that they are essentially unique up to symmetries and reordering of terms.

As a first step in this direction, we could check via computer search that Theorem \ref{thmbox3} holds with equality when $2 \le m \le 7$. In Section \ref{d=3}, we will also show, using results on discrete Euclidean balls proved in \cite{GP26}, that $\DD^{(4)} ( \llbracket -m,m \rrbracket ^3)$ asymptotically matches the lower bound given by Theorem \ref{thmbox3} as $m$ tends to infinity. 

\begin{theorem}
\label{thmbox3bis}
Asymptotically, as $m \to + \infty$,
$$
\DD^{(4)} (\llbracket -m,m \rrbracket ^3) \sim  16 m^3.
$$
\end{theorem}

It appears tempting to conjecture that for every integer $d \ge 1$, the ratio $\DD^{(d+1)} (\llbracket -m,m \rrbracket ^d)/m^d$ has a limit when $m$ tends to infinity. This is clear when $d=1$, and it follows from Theorems \ref{th1} and \ref{thmbox3bis} when $d \in \{2,3\}$. However, this problem remains open when $d \ge 4$. 

\section{Two-dimensional boxes}
\label{sectThm4}

We start by proving the lower bounds in Theorem \ref{th1} and Corollary \ref{improvement44}. To do so, it suffices to show that the sequence $S(m)$, which was defined in Section \ref{newresults} just before the statement of Theorem \ref{th1}, is a minimal zero-sum sequence. 

First, $S(m)$ is a zero-sum sequence since
$$\sigma(S(m))=
(q(m)(m-1)+m) 
\begin{pmatrix}
m \\
m 
\end{pmatrix}
+ (2m^2-mq(m))
\begin{pmatrix}
-m \\
m-1 
\end{pmatrix}
+ (2m^2-m)
\begin{pmatrix}
m-q(m) \\
-m
\end{pmatrix}
= 
\begin{pmatrix}
0 \\
0
\end{pmatrix}.
$$

Now, let us show that $S(m)$ is minimal. 
Assume that $k_1,k_2,k_3$ are three non-negative integers satisfying 
\begin{equation}
\label{bound0}
k_1 \leq q(m)(m-1)+m,\quad  k_2 \leq 2m^2-mq(m) \quad {\rm and }\quad  k_3 \leq 2m^2-m,
\end{equation}
and such that the sequence consisting of $k_1$ copies of $\begin{pmatrix}
m \\
m 
\end{pmatrix}$, $k_2$ copies of $\begin{pmatrix}
-m \\
m-1 
\end{pmatrix}$ and $k_3$ copies of $\begin{pmatrix}
m-q(m) \\
-m
\end{pmatrix}$ is a zero-sum sequence.

These three integers must satisfy the system
\begin{equation}
\label{2coord}
\begin{cases}
k_1m-k_2m+k_3(m-q(m))=0, \\
k_1m+k_2(m-1)-k_3m=0.
\end{cases}
\end{equation}
Adding $(m-1)$ times the first equation of \eqref{2coord} to $m$ times the second one yields

\begin{equation}
\label{k1}
k_1m(2m-1) = k_3(q(m)(m-1)+m).
\end{equation}

We observe that $\gcd (m(2m-1), q(m)(m-1)+m)=1$. Indeed, let $p$ be any prime dividing this gcd.
It is easily noticed that $p \nmid m$, since otherwise we would have $p \mid q(m)(m-1)$ so that $p \mid q(m)$ and finally $p \mid \gcd(m,q(m))=1$, a contradiction. 
Therefore, $p \mid 2m-1$ so that
$$p \mid (q(m)(m-1)+m)-(2m-1)=(q(m)-1)(m-1).$$

It follows that $p \mid m-1$ since otherwise, we would have $p \mid q(m)-1$ which implies that $p < q(m)$ that is to say $p \mid m$, a contradiction. Finally, we obtain $p \mid 2m-1-(m-1)=m$, a contradiction once again.

As a conclusion, \eqref{bound0}, \eqref{k1} and the just-proved coprimality of $m(2m-1)$ and $q(m)(m-1)+m$ imply that 
$$m(2m-1) \mid k_3 \le m(2m-1),$$ 
so that either $k_3=0$ in which case $k_1=0$ so that $k_2=0$, or $k_3=m(2m-1)$ in which case $k_1=q(m)(m-1)+m$ so that $k_2=2m^2-mq(m)$. 

This proves our assertion on the minimality of $S(m)$, and gives
$$
\DD (\llbracket -m,m\rrbracket^2) \geq \DD^{(3)} (\llbracket -m,m\rrbracket^2) \geq  \|S(m) \|  = 4m^2-q(m),$$
which is the desired lower bound of Theorem \ref{th1} and Corollary \ref{improvement44}.

We continue with a useful lemma on the arithmetic function $q$ which clarifies the asymptotic behaviour of this lower bound. 
Recall that, for every integer $m \ge 2$, the number $q(m)$ denotes the smallest prime number not dividing $m$. 

\begin{lemma}
\label{qq}
For every integer $m \geq 2$, one has
$$
q(m) \leq 1 + 4 \log m.
$$
Moreover, if $m \geq 3$, then
$$
q(m)<m.
$$
\end{lemma}

\begin{proof}
Assume that $q(m) \leq 10$. Since it is a prime number, $q(m)$ can be equal to $2,3,5$ or 7.
If $q(m)=2$, $m \ge 3$; if $q(m)=3$, $m \ge 2$; if $q(m)=5$, $m \ge 6$ and 
if $q(m)=7$, $m \ge 30$. In these four cases, an immediate computation shows that 
the inequality is valid.

Let us now assume that $q(m) \geq 11$.
By definition, the product 
$${\displaystyle \prod_{p<q(m),\ p\, {\rm prime}} p}$$ 
divides $m$, thus, 
using the classical number-theoretic notation $\vartheta$ for the Chebyshev function \cite{HW}, we obtain
$$
\vartheta (q(m)-1) = \sum_{p<q(m),\ p\, {\rm prime}} \log p \leq \log m.
$$ 
Dusart proved (this is Theorem 4.2 in \cite{Dusart}) that
$$
 \forall x \geq 2 :\quad \quad  \vartheta (x) \geq x - 3.965 \frac{x}{\log^2 x}.
$$
We therefore have
$$
\frac14 (q(m)-1) \leq \left( 1 - \frac{3.965}{\log^2 10} \right) (q(m)-1) \leq \vartheta (q(m)-1) \leq \log m,
$$
due to the fact that $q(m)$ is by definition larger than $10$. The first assertion follows.

Finally, the second assertion follows from the first if $m \geq 11$ and from a direct computation if $3 \leq m \leq 10$.

\end{proof}

We now turn to the much more involved proof of the upper bound in Theorem \ref{th1}.

The case $m=2$ can be considered independently since an exhaustive computer search is possible. It will 
simplify the presentation of the forthcoming argument (mainly due to the fact that the second part of Lemma \ref{qq} 
cannot be applied when $m=2$).
One checks that
$$
\DD^{(3)} (\llbracket -2,2 \rrbracket ^2) = 13 = 4 \times 2^2 -3,
$$
and that, up to the eight symmetries of the square (generated by $\varphi : (x,y) \mapsto (-x,y)$ and $\psi : (x,y) \mapsto (y,x)$) and reordering of terms, the only minimal zero-sum sequence over $\llbracket -2,2 \rrbracket ^2$ having a support of size $3$ and length $13$ is the one consisting of 

$$5 \text{ copies of } \begin{pmatrix}
2 \\
2
\end{pmatrix},
\ 2 \text{ copies of } \begin{pmatrix}
-2 \\
1
\end{pmatrix}, 
\ \text{ and } 6 \text{ copies of } \begin{pmatrix}
-1 \\
-2
\end{pmatrix}
.$$

Since $q(2)=3$, the result is proved for $m=2$.

From now on, we assume $m \geq 3$. We start by a reformulation of the problem for which we need some notation.
Let 
$$
A =\left(\begin{array}{ccc}
a_{1,1} &   a_{1,2} & a_{1,3}  \\ 
a_{2,1} &   a_{2,2} & a_{2,3}  \\ 	
\end{array} 
\right)
$$
be a $2 \times 3$ matrix with coefficients in $\llbracket -m, m \rrbracket$. 
For $i=1,2$ or $3$, we denote by $A_i$ the $2 \times 2$ submatrix of $A$ obtained by erasing the $i$-th column
and let 
$$
\Delta (A)=\gcd \big( \det(A_1),\det(A_2),\det(A_3) \big).
$$ 
From now on, we define ${\mathcal C}$ to be the set of $2 \times 3$ matrices with coefficients in $\llbracket -m, m \rrbracket$ 
such that $\det(A_1),\det(A_2),$ $\det(A_3) \neq 0$ and $\Delta (A)=1$.

Finally, we shall say that two  $2 \times 3$ matrices $A$ and $B$ with coefficients in $\llbracket -m, m \rrbracket$ are 
{\em analogous} if $B$ can be deduced from $A$ by exchanging columns and applying one of the eight symmetries of the square 
already mentioned (those generated by $\varphi : (x,y) \mapsto (-x,y)$ and $\psi : (x,y) \mapsto (y,x)$). Notice that 
two analogous matrices $A$ and $B$ satisfy $\Delta (A) = \Delta (B)$.

\begin{lemma}
Let $m \geq 3$ and $S$ be a minimal zero-sum sequence with a support of size 3, $\{s_1, s_2, s_3 \} \subseteq \llbracket -m, m \rrbracket^2$, 
such that $\DD^{(3)} (\llbracket -m,m \rrbracket ^2)  = \|S\|$.
Let $A$ be the matrix, the $i$-th column of which is $s_i$ ($i=1,2,3$), then we have $A  \in  {\mathcal C}$ and
$$
\DD^{(3)} (\llbracket -m,m \rrbracket ^2)  = \sum_{i=1}^3 | \det A_i |.    
$$
\end{lemma}

\begin{proof}
Firstly, we notice that if $A$ has rank 1, then the three columns of $A$ are collinear. 
In this case, the length of our minimal zero-sum sequence $S$ with support $\{  s_1, s_2, s_3\}$ 
would be at most the upper bound obtained for the one-dimensional case, which is $2m-1$. 
Yet, it follows from Theorem \ref{th1} that
$$
\DD^{(3)} (\llbracket -m,m \rrbracket ^2)  \geq 4m^2 -q(m)
$$
which gives a contradiction. We may therefore assume that $A$ has rank 2.
Moreover, we must have 
\begin{equation}
\label{detnonnul}
\det(A_1),\det(A_2),\det(A_3) \neq 0
\end{equation}
for, if say $\det(A_1)=0$ then $s_2$ and $s_3$ are proportional but in this case, 
in any linear combination of $s_1, s_2$ and $s_3$ equal to zero, the coefficient of $s_3$ itself must be zero 
and we are back again to the one-dimensional case.

By Cramer's method, it is readily seen that the  one-dimensional kernel of $A$ is equal to the set of multiples of 
$$
r = \frac{1}{\Delta (A)}
\begin{pmatrix}
\det(A_1) \\
 -\det(A_2)\\
 \det(A_3)
\end{pmatrix}
$$
a vector with integral coefficients, and indeed the one with the smallest integral coefficients. 
In particular, all the elements of $(\ker A ) \cap \ZZ^3$ are integral multiples of $r$.
Denoting by $\ell_A$ the smallest  $\ell_1$-norm of a non-zero  
integral solution $X$ to the system $AX=0$, we find 
$$
\ell_A = \frac{\sum_{i=1}^3 | \det A_i |}{\Delta (A)},
$$
and we observe that $\ell_A$ is the length of our minimal zero-sum sequence $S$.
Thus, by assumption,
\begin{equation}
\label{D3par les det}
\DD^{(3)} (\llbracket -m,m \rrbracket ^2)  = \ell_A = \frac{\sum_{i=1}^3 | \det A_i |}{\Delta (A)}.
\end{equation}

Again, by the lower bound we obtained in the previous section (see Theorem \ref{th1}), we know that
\begin{equation}
\label{rho}
\ell_A  \geq 4m^2 -q(m).
\end{equation}
Since for $i=1,2,3$,  $ | \det A_i |\leq 2 m^2$, we obtain 
$$
\ell_A \leq  \frac{6m^2}{\Delta (A)} 
$$
and using \eqref{rho}, we get
$$
\Delta (A) = \gcd_{i \in \{1,2,3 \} }( \det A_i) \leq \frac{6m^2}{4m^2 -q(m)} <2
$$
in view of $q(m) \leq m <m^2$ (for $m \geq 3$),  by Lemma \ref{qq}.
From this we deduce  
$$
\Delta (A) = \gcd_{i \in \{1,2,3 \} }  ( \det A_i)=1.
$$
This, with \eqref{detnonnul}, proves that $A \in {\mathcal C}$. Moreover, by \eqref{D3par les det},
$$
\DD^{(3)} (\llbracket -m,m \rrbracket ^2)  = \sum_{i=1,2,3} | \det A_i |.
$$
\end{proof}

Let us introduce ${\mathcal C}'$ the subset of ${\mathcal C}$ consisting of matrices $A$ satisfying the sign pattern given by
$$
\left(\begin{array}{ccc}
+& -& +\\ 
+&+&-  \\ 	
\end{array} 
\right).
$$
We show how to restrict our research on this subset of ${\mathcal C}$.

\begin{lemma}
\label{lelele6}
Let $m \geq 3$ and $S$ be a minimal zero-sum sequence with a support of size 3, $\{s_1, s_2, s_3 \} \subseteq \llbracket -m, m \rrbracket^2$, such that 
$\DD^{(3)} (\llbracket -m,m \rrbracket ^2)  = \|S\|$.
Let $A$ be the matrix, the $i$-th column of which is $s_i$ ($i=1,2,3$).
Then $A$ is analogous to a matrix $B$ belonging to ${\mathcal C'}$ and
$$
\DD^{(3)} (\llbracket -m,m \rrbracket ^2)  = \sum_{i=1}^3 | \det B_i |.   
$$
\end{lemma}

We notice that two analogous matrices $A$ and $B$ satisfy $\sum_{i=1}^3 | \det A_i | = \sum_{i=1}^3 | \det B_i |$.
Lemma \ref{lelele6} expresses the fact that in order to compute the maximum of $\sum_{i=1}^3 | \det A_i |$ over ${\mathcal C}$
it is enough to restrict ourselves to matrices belonging to ${\mathcal C}'$ and that every matrix of ${\mathcal C}$ 
attaining the value $\DD^{(3)} (\llbracket -m,m \rrbracket ^2)$ is, up to a reordering of the columns and symmetries, 
a matrix of ${\mathcal C}'$.

\begin{proof}
We denote
$$
A = \left(\begin{array}{ccc}
a_{1,1} &   a_{1,2} & a_{1,3}  \\ 
a_{2,1} &   a_{2,2} & a_{2,3}  \\ 	
\end{array} 
\right).
$$

By permutation of the columns, we may assume that $\max_{i=1,2,3} | \det A_i | =| \det A_3 |$ and 
by the symmetry of coordinates and permutation of the vectors
$$
a_{1,1} , a_{2,1}  \geq 0.
$$
Such a symmetry is tantamount to multiplying by $-1$ a line or a column, 
which does not change the absolute values of the determinants
and is compatible with our assumption that $| \det A_3 | = \max_{i \in \{ 1,2,3\}} | \det A_i |$.

By the zero-sum sequence property, there is a sum of the three columns of $A$ with positive coefficients which is equal to zero. 
Thus, we must have $ a_{1,2}$ or $ a_{1,3} \leq 0$. 
We may assume without loss of generality that 
$$
a_{1,2}  \leq 0.
$$

Since by assumption $\DD^{(3)} (\llbracket -m,m \rrbracket ^2)  \geq 4m^2 -q(m)$  and $q(m) \leq  m-1 \leq  m^2 / 3 $ 
(for any $m \geq 3$, this bound follows from Lemma \ref{qq})		
$$
| \det A_3 | = \max_{i \in \{ 1,2,3\}} | \det A_i | \geq \frac{\sum_{i=1,2,3} | \det A_i |}{3} = 
\frac{\DD^{(3)} (\llbracket -m,m \rrbracket ^2) }{3} \geq  \frac{4m^2 -q(m) }{3}  \geq    \frac{11}{9}\ m^2.
$$
This yields
$$
| a_{1,1} a_{2,2} - a_{2,1}a_{1,2}| \geq \frac{11}{9} \ m^2 > m^2
$$
and the two terms $a_{1,1} a_{2,2}$ and $a_{2,1}a_{1,2}$ must be of opposite sign.
As $a_{1,2}a_{2,1}\leq 0$ and $a_{1,1} \geq 0$, we deduce 
$$
a_{2,2} \geq 0.
$$

Considering the second coordinate of the three terms of the support implies, by the zero-sum sequence property, 
$$
a_{2,3}  \leq 0.
$$

For the matrix $A$, we thus have only two possible combinations of signs
$$
\left(\begin{array}{ccc}
+& -& - \\ 
+&+&-  \\ 	
\end{array} 
\right)
$$
or
$$
\left(\begin{array}{ccc}
+& -& +\\ 
+&+&-  \\ 	
\end{array} 
\right)
$$
depending on the sign of $a_{1,3}$.

In the first case, exchanging the first two vectors and applying a symmetry (the first coordinate is replaced by its opposite), 
we arrive in the second case situation. Therefore, without restriction, we may assume that we are in the case
$$
\left(\begin{array}{ccc}
+& -& +\\ 
+&+&-  \\ 	
\end{array} 
\right).
$$

Denoting by $B$ the matrix obtained after all the transformations we performed from the beginning, starting from $A$, 
we obtain a matrix $B$ analogous to $A$ such that $B \in {\mathcal C}'$ and 
$$
\DD^{(3)} (\llbracket -m,m \rrbracket ^2)  = \sum_{i=1}^3 | \det A_i |=  \sum_{i=1}^3 | \det B_i |.    
$$
\end{proof}

In view of the two previous lemmas, computing $\DD^{(3)} (\llbracket -m,m \rrbracket ^2)$ consists in 
maximising $\sum_{i=1}^3 | \det A_i |$ on matrices of ${\mathcal C}'$ :
$$
\DD^{(3)} (\llbracket -m,m \rrbracket ^2) = \max_{A \in {\mathcal C}'} \sum_{i=1}^3 | \det A_i |.
$$
Using the sign pattern of ${\mathcal C}'$, we may reorganise  the sum of determinants appearing in this maximum: 
the function we have to maximise for 
$A \in{\mathcal C}'$ is
\begin{eqnarray}
&  =  &  |a_{1,1} a_{2,2} - a_{1,2}a_{2,1}|+ |a_{1,3} a_{2,1} - a_{2,3}a_{1,1}| +    |a_{2,3}a_{1,2}-a_{1,3} a_{2,2} |  \nonumber \\
	&  =  &    ( a_{2,2} - a_{2,3}) a_{1,1}  +       (a_{1,3}   - a_{1,2} )     a_{2,1} 	   +      |a_{2,3}a_{1,2}-a_{1,3} a_{2,2} | \label{defdef}
\end{eqnarray}
We define this quantity as $f (a_{1,1} ,a_{1,2} ,a_{1,3},a_{2,1} , a_{2,2} , a_{2,3})$.

We now fix $M \in {\mathcal C}'$, a matrix attaining the maximum
$\max_{A \in {\mathcal C}'}\sum_{i=1}^3 | \det A_i | = \DD^{(3)} (\llbracket -m,m \rrbracket ^2) $ and write
\begin{equation}
\label{defofM}
M=\left(\begin{array}{ccc}
x_{1} &   y_{1} & z_{1}  \\ 
x_{2} &   y_{2} & z_{2}  \\ 	
\end{array} 
\right).
\end{equation}
We derive some properties on it. For this purpose, we denote by
$$
{\mathcal M}^{\ast}_{2,3}(\llbracket -m,m \rrbracket ^2)
$$
the set of all $2 \times 3$ matrices with coefficients in $\llbracket -m, m \rrbracket$ that satisfy the same sign pattern
$$
\left(\begin{array}{ccc}
+& -& +\\ 
+&+&-  \\ 	
\end{array}
\right)
$$
as the elements in ${{\mathcal C}'}$, but not necessarily the other conditions.

\begin{lemma}
\label{lelemme1}
One has $x_{1}  =m$.
\end{lemma}

\begin{proof}
Otherwise, $x_{1} \leq m-1$ and
\begin{eqnarray*}
 \DD^{(3)} (\llbracket -m,m \rrbracket ^2)    & =     & f (x_{1}, y_{1}, z_{1}, x_{2}, y_{2}, z_{2}) \\
   & =   & \max_{A \in {\mathcal C}',\ 
   a_{1,1} \leq m-1}  f (a_{1,1} ,a_{1,2} ,a_{1,3},a_{2,1} , a_{2,2} , a_{2,3})\\
   & \leq & \max_{A \in {\mathcal M}^{\ast}_{2,3}(\llbracket -m,m \rrbracket ^2),\  a_{1,1} \leq m-1}  f (a_{1,1} ,a_{1,2} ,a_{1,3},a_{2,1} , a_{2,2} , a_{2,3})\\
    & \leq & \max_{A \in {\mathcal M}^{\ast}_{2,3}(\llbracket -m,m \rrbracket ^2) ,\  a_{1,1} = m-1}  f (a_{1,1} ,a_{1,2} ,a_{1,3},a_{2,1} , a_{2,2} , a_{2,3})
\end{eqnarray*}
in view of the fact that the coefficient of $a_{1,1}$ in $f$ (see \eqref{defdef}) is $a_{2,2}  - a_{2,3} \geq 0$,
by our rule of signs on ${\mathcal M}^{\ast}_{2,3}(\llbracket -m,m \rrbracket ^2)$. 
Notice that there is no arithmetic constraint anymore in this maximum. 
Thus, whatever the value of $a_{2,2}  - a_{2,3}$  is, the maximum is attained for 
$$
a_{1,1}=m-1.
$$ 
Then, the coefficient of $a_{2,1}$  in the maximum is $a_{1,3}  - a_{1,2} \geq 0$ (see \eqref{defdef} and the sign pattern again) 
which is therefore attained on ${\mathcal M}^{\ast}_{2,3}(\llbracket -m,m \rrbracket ^2)$ for 
$$
a_{2,1}=m.
$$ 
Thus
$$
\DD^{(3)} (\llbracket -m,m \rrbracket ^2)   \leq  
\max_{ A \in {\mathcal M}^{\ast}_{2,3}(\llbracket -m,m \rrbracket ^2),\  a_{1,1} = m-1, a_{2,1}=m} ( a_{2,2} - a_{2,3}) (m-1) + (a_{1,3} - a_{1,2} ) m + |a_{2,3}a_{1,2}-a_{1,3} a_{2,2} |.
$$

If $a_{2,3}a_{1,2}-a_{1,3} a_{2,2} \geq 0$, then in the function we have to maximise
\begin{eqnarray*}
&&( a_{2,2} - a_{2,3}) (m-1) + (a_{1,3} - a_{1,2} ) m + |a_{2,3}a_{1,2}-a_{1,3} a_{2,2} | \\
&=& ( a_{2,2} - a_{2,3}) (m-1) + (a_{1,3} - a_{1,2} ) m + a_{2,3}a_{1,2}-a_{1,3} a_{2,2} \\
&=& (a_{2,3}  - m) a_{1,2} +( a_{2,2} - a_{2,3}) (m-1)+ a_{1,3} (m-a_{2,2})
\end{eqnarray*}
the coefficient of $a_{1,2}$ is $a_{2,3}  - m \leq -m < 0$. Therefore the maximum is attained for 
$$
a_{1,2}=-m,
$$ 
while, the coefficient of $a_{2,3}$ is $a_{1,2}  - (m-1) =- (2m-1) < 0$, so that the maximum is attained for 
$$
a_{2,3}=-m.
$$ 
Finally, it remains to maximise the quadratic function in $a_{2,2}$ and $a_{1,3}$:
\begin{eqnarray*}
&&(m-1)a_{2,2} +m^2 + ma_{1,3}  +m^2 -m +m^2  -a_{1,3} a_{2,2} \\
&=& 3m^2-m +(m-1)a_{2,2} +ma_{1,3}   -a_{1,3} a_{2,2} ,
\end{eqnarray*}
which is maximal for $a_{1,3} =m$ and  $a_{2,2}=0$. As a consequence, the upper bound we obtain for 
$\DD^{(3)} (\llbracket -m,m \rrbracket ^2) $ is 
$$
3 m^2 -m+m^2=4m^2 -m,
$$
a contradiction with the lower bound of Theorem \ref{th1} (for $m \geq 3$) using the inequality $q(m)<m$ of Lemma \ref{qq}.

In the case $a_{2,3}a_{1,2}-a_{1,3} a_{2,2} \leq 0$, we proceed in the same way. 

\end{proof}

\begin{lemma}
\label{lelemme222}
One has $x_{2}  =m$.
\end{lemma}

\begin{proof}
If our conclusion is false, then $x_{2}  \leq m-1$.
Since $x_{1} = m$ by Lemma \ref{lelemme1}, we have
\begin{eqnarray*}
\DD^{(3)} (\llbracket -m,m \rrbracket ^2)  & =     &   f (x_{1}, y_{1}, z_{1}, x_{2}, y_{2}, z_{2}) \\
							& = &  f (m , y_{1}, z_{1}, x_{2}, y_{2}, z_{2})\\
							   & =   & \max_{A \in {\mathcal C}',\ 
   \ a_{1,1} =m,\ a_{2,1}  \leq m-1}  f (m ,a_{1,2} ,a_{1,3},a_{2,1} , a_{2,2} , a_{2,3})\\
    & \leq & \max_{A \in {\mathcal M}^{\ast}_{2,3} (\llbracket -m,m \rrbracket ^2),\  a_{1,1} = m,\ a_{2,1}  =m-1}  f (m , a_{1,2} ,a_{1,3},m-1, a_{2,2} , a_{2,3}).
\end{eqnarray*}
Indeed, the coefficient of $a_{2,1}$  in the maximum \eqref{defdef} is $a_{1,3}  - a_{1,2} \geq 0$ 
which is therefore attained for a maximal $a_{2,1}$, that is
$$
a_{2,1}=m-1.
$$ 
Thus
$$
\DD^{(3)} (\llbracket -m,m \rrbracket ^2)  \leq  
\max_{ A \in {\mathcal M}^{\ast}_{2,3},\  a_{1,1} = m,\ a_{2,1}=m-1}  
m( a_{2,2} - a_{2,3}) + (m-1)(a_{1,3} - a_{1,2} ) + |a_{2,3}a_{1,2}-a_{1,3} a_{2,2} |.
$$

If $a_{2,3}a_{1,2}-a_{1,3} a_{2,2} \geq 0$, then in
$$
 m( a_{2,2} - a_{2,3})  +       (m-1) (a_{1,3}   - a_{1,2} ) 	   +      a_{2,3}a_{1,2}-a_{1,3} a_{2,2}
$$
the coefficient of $a_{1,2}$ is $a_{2,3}  - (m-1) \leq -(m-1) < 0$. Therefore the maximum is attained for 
$$
a_{1,2}=-m,
$$ 
while, the coefficient of $a_{2,3}$ is $a_{1,2}  - m= - 2m < 0$, so that the maximum is attained for 
$$
a_{2,3}=-m.
$$ 
Finally, it remains to maximise the quadratic function in $a_{2,2}$ and $a_{1,3}$:
\begin{eqnarray*}
&& ma_{2,2} +m^2 + (m-1) a_{1,3}  +m^2 -m +m^2  -a_{1,3} a_{2,2} \\
&=& 3m^2-m +m a_{2,2} +(m-1) a_{1,3}   -a_{1,3} a_{2,2} ,
\end{eqnarray*}
which is maximal for $a_{2,2}=m$  and $a_{1,3} =0$. Finally the upper bound we obtain is
$$
\DD^{(3)} (\llbracket -m,m \rrbracket ^2)  \leq 3 m^2 -m+m^2=4m^2 -m,
$$
a contradiction with the lower bound of Theorem \ref{th1} (for $m \geq 3$) using the inequality $q(m)<m$ of Lemma \ref{qq}.

In the case $a_{2,3}a_{1,2}-a_{1,3} a_{2,2} \leq 0$, we proceed in the same way. 

\end{proof}

\begin{lemma}
\label{lelemme3}
One has $y_{1}  =z_{2}=-m$.
\end{lemma}

\begin{proof}
If our conclusion is false, then $y_{1}+z_{2}  \geq -(2m-1)$ and we have by Lemmas \ref{lelemme1} and \ref{lelemme222} 
\begin{eqnarray*}
\DD^{(3)} (\llbracket -m,m \rrbracket ^2) & =     &   f (x_{1}, y_{1}, z_{1}, x_{2}, y_{2}, z_{2}) \\
							& =     &   f (m , y_{1}, z_{1}, m, y_{2}, z_{2})\\
				   & =   & \max_{A \in {\mathcal C}',\  a_{1,1} = a_{2,1} =m}  f (m ,a_{1,2} ,a_{1,3}, m , a_{2,2} , a_{2,3})\\
  				 & \leq & \max_{A \in {\mathcal M}^{\ast}_{2,3} (\llbracket -m,m \rrbracket ^2),\  a_{1,1} = a_{2,1}  =m,\ a_{1,2}+a_{2,3}  \geq -(2m-1)}  
				 								f (m ,a_{1,2} ,a_{1,3},m , a_{2,2} , a_{2,3}).
\end{eqnarray*}

If $a_{2,3}a_{1,2}-a_{1,3} a_{2,2} \geq 0$, then
\begin{eqnarray*}
& &  f (m ,a_{1,2} ,a_{1,3},m , a_{2,2} , a_{2,3})\\
&=& m( a_{2,2} - a_{2,3})  +      m (a_{1,3}   - a_{1,2} )  	   +      a_{2,3}a_{1,2}-a_{1,3} a_{2,2} \\
&=& m( a_{2,2} + a_{1,3} - (a_{1,2}+a_{2,3} ))+ a_{2,3}a_{1,2}-a_{1,3} a_{2,2}.
\end{eqnarray*}
In this function, the coefficients of $a_{1,3}$ and $a_{2,2}$ are respectively $m-a_{2,2}$ and $m-a_{1,3}$ which are both non-negative.
Therefore the maximum is attained for 
$$
a_{1,3}=a_{2,2}=m.
$$ 
It follows, using $a_{1,2}+a_{2,3} \geq -(2m-1)$, that
\begin{eqnarray*}
 f (m ,a_{1,2} ,a_{1,3},m , a_{2,2} , a_{2,3})& \leq  &  m( 2m - (a_{1,2}+a_{2,3} ))+ a_{2,3}a_{1,2}-m^2\\
   & \leq  & m(4m-1)-m^2 + a_{2,3}a_{1,2}\\
   &\leq & 3m^2 - m+m^2 = 4m^2 -m,
 \end{eqnarray*}
a contradiction with the lower bound of Theorem \ref{th1} (for $m \geq 3$) using the inequality $q(m)<m$ of Lemma \ref{qq}.

In the case $a_{2,3}a_{1,2}-a_{1,3} a_{2,2} \leq 0$, we proceed in the same way. 
\end{proof}

Finally, using the previous three lemmas, and going back to Lemma \ref{lelele6} and the definition \eqref{defofM} of $M$, the matrix 
attaining the maximum we are computing,  we obtain
\begin{equation}
\label{eqfinale}
\DD^{(3)} (\llbracket -m,m \rrbracket ^2)  =   m( y_{2} + z_{1} + 2m ) + | m^2-z_{1} y_{2} | =m( y_{2} + z_{1} + 2m ) + (m^2-z_{1} y_{2} ),
\end{equation}
on recalling that the to unknown coefficients $y_{2}$ and $z_{1}$ satisfy together with the other coefficients 
of the matrix $M$ the arithmetic conditions included in the definition of ${\mathcal C}'$.
We may indeed rewrite the arithmetic condition $\Delta (M)=1$ which appears in the definition of $ {\mathcal C}' $ as
$$
\gcd ( m^2 + my_{2}, m^2 + mz_{1}, m^2 - y_{2}z_{1} ) = 1
$$
which implies that
$$
\gcd (m, z_{1} ) =  \gcd (m, y_{2} ) =1\quad \text{ and }\quad  z_{1}  \neq y_{2}  .
$$
Indeed, if an integer $d$ divides for instance $m$ and $y_{2}$ then, 
it divides $\gcd ( m^2 + my_{2}, m^2 + mz_{1}, m^2 - y_{2}z_{1} )=1$ thus $d=1$. 
Moreover, if $z_{1}$ could be equal to $y_{2}$, say $a$, 
then, $m+a \geq m >1$ would divide all three terms in the above gcd, a contradiction.  
Thus $z_{1}$ and $y_{2}$ must be two distinct integers coprime to $m$. 

But the coefficient of $y_{2} $ in the right-hand side of  \eqref{eqfinale} is $m-z_{1} \geq 0$ and the one of 
$z_{1} $ is $m-y_{2} \geq 0$, thus to maximise this polynomial we have to have as big as possible values for these 
two variables. It turns out that the biggest integer coprime to $m$ and less than $m$ is $m-1$ and that the biggest 
integer coprime to $m$ and less than $m-1$ is $m-q(m)$. Finally we obtain replacing $z_{1}$ by $m-1$ and $y_2$ by 
$m-q(m)$ or conversely (the function is symmetric in these two variables) :
\begin{eqnarray*}
\DD^{(3)} (\llbracket -m,m \rrbracket ^2)  & \leq & 3m^2 + m \big( m-1+m-q(m) \big)  - (m-1)(m-q(m)) \\
								& = & 4m^2 -q(m),
\end{eqnarray*}
which concludes the proof of the upper bound in Theorem \ref{th1}.

The method used above is indeed constructive and gives directly the `moreover' statement. 
Notice that, at the end of the proof, we may have exchange the values given to $y_{2}$ and $z_{1}$, 
but the two corresponding matrices are analogous in the sense we defined above. 
\smallskip

A final remark: proving that $\DD^{(3)} (\llbracket -m,m \rrbracket ^2)  \leq 4 m^2$ 
is a significantly easier task. Indeed, obtaining such a weaker upper bound does not require to take the arithmetic constraints into consideration.

\section{Three-dimensional boxes}
\label{d=3}

We start by proving Theorem \ref{thmbox3}.

\begin{proof}[Proof of Theorem \ref{thmbox3}]

Let $m \ge 2$ be an integer. We consider the cases $m$ odd and $m$ even separately.

Assume first that $m$ is odd and define the sequence $T(m)$
consisting of
$\begin{pmatrix}
- m \\
- m \\
- m \\
\end{pmatrix}
$
repeated $4m^3-8m^2+5m-2$ times, 
$\begin{pmatrix}
- m \\
m -1\\
m \\
\end{pmatrix}
$
repeated $4m^3-2m^2+2m$ times, 
$\begin{pmatrix}
m\\
-m \\
m-2\\ 
\end{pmatrix}$  
repeated $4m^3-2m^2+m$ times and finally
$
\begin{pmatrix}
m-1 \\
m\\
-m 
\end{pmatrix} 
$
repeated $4m^3-4m^2+2m$ times. We may compute that
\begin{eqnarray*}
\|T(m)\| & = & (4m^3-8m^2+5m-2)+(4m^3-2m^2+2m)+(4m^3-2m^2+m)+(4m^3-4m^2+2m) \\
	&=& 16 m^3 -16m^2 + 10m -2.
\end{eqnarray*}

It is a zero-sum sequence since
\begin{eqnarray*}
 &\hspace{-6.5cm}(4m^3-8m^2+5m-2)
\begin{pmatrix}
- m \\
- m \\
- m \\
\end{pmatrix}
+
(4m^3-2m^2+2m)
\begin{pmatrix}
- m \\
m -1\\
m \\
\end{pmatrix}
& \\
& 
\hspace{5cm}+(4m^3-2m^2+m)
\begin{pmatrix}
m\\
-m \\
m-2\\ 
\end{pmatrix} 
+ (4m^3-4m^2+2m)
\begin{pmatrix}
m-1 \\
m\\
-m 
\end{pmatrix} 
=
\begin{pmatrix}
0 \\
0\\
0 
\end{pmatrix}. &
\end{eqnarray*}

To show that it is a minimal zero-sum sequence, we solve 
$$
k_1
\begin{pmatrix}
- m \\
- m \\
- m \\
\end{pmatrix}
+
k_2
\begin{pmatrix}
- m \\
m -1\\
m \\
\end{pmatrix} 
+k_3
\begin{pmatrix}
m\\
-m \\
m-2\\ 
\end{pmatrix} 
+ k_4
\begin{pmatrix}
m-1 \\
m\\
-m 
\end{pmatrix} 
=
\begin{pmatrix}
0 \\
0\\
0 
\end{pmatrix}
$$
for non-negative integers $k_1,k_2,k_3,k_4$ that are not all zero and satisfy
\begin{eqnarray}
\nonumber
k_1 \leq 4m^3-8m^2+5m-2, && k_2 \leq 4m^3-2m^2+2m,\\
\label{xis}
k_3 \leq 4m^3-2m^2+m, &\quad {\rm and } \quad& k_4 \leq  4m^3-4m^2+2m.
\end{eqnarray}

This implies
$$
k_4 (m-1) = (k_1+k_2-k_3)m,\ k_2 (m-1) = (k_1+k_3-k_4)m,\ k_3 (m-2) = (k_1-k_2+k_4)m.
$$
Therefore, there are three integers $l_1,l_2$ and $l_3$ such that 
\begin{eqnarray}
k_4 & = & l_1\ m \label{eq1} \\
k_1+k_2-k_3 & = & l_1\ (m-1) \label{eq2} \\
k_2 & = & l_2\ m \label{eq3} \\
k_1+k_3-k_4 & = & l_2\ (m-1) \label{eq4} \\
k_3 &=& l_3\ m \label{eq5} \\
k_1-k_2+k_4 & = & l_3\ (m-2) \label{eq6}
\end{eqnarray}
Summing \eqref{eq2} and \eqref{eq4} gives $2k_1 +k_2 -k_4 = (l_1+l_2) (m-1)$ which, with \eqref{eq1} and \eqref{eq3}, yields
\begin{equation}
\label{eqx1}
k_1 = \frac{(l_1+l_2) (m-1) + (l_1-l_2)m}2.
\end{equation}
In an analogous way, one gets $2k_3 -k_2 -k_4 = (l_2-l_1) (m-1)$ and then
\begin{equation}
\label{eqx3}
k_3 = \frac{(l_1+l_2) m + (l_2-l_1)(m-1)}2.
\end{equation}
Replacing in \eqref{eq5}, we obtain
\begin{equation}
\label{eqA}
(l_1+l_2-2l_3) m = (l_1-l_2)(m-1).
\end{equation}
With \eqref{eq6}, we get $(l_1+l_2) (m-1) + 3(l_1-l_2)m =2l_3 (m-2)$ and after replacing $2(m-1)$ by $m+(m-2)$, 
\begin{equation}
\label{eqB}
(7 l_1 -5l_2) m = (4l_3 -l_1-l_2)(m-2).
\end{equation}

Equation \eqref{eqA} gives, using the coprimality condition $\gcd (m, m-1)=1$, the existence of an integer $q$ such that 
$$
l_1 -l_2 =qm,\quad {\rm and}\quad l_1+l_2-2l_3=q(m-1),
$$
and \eqref{eqB} gives the existence of an integer $r$ such that 
$$
7 l_1 -5l_2 =r(m-2) ,\quad {\rm and}\quad 4l_3 -l_1-l_2 =rm.
$$
With the two first equalities and the fourth one, one solves in $(l_1,l_2,l_3)$, namely: 
$$
l_1 = \frac{(3qm-2q+rm)}{2}, \quad l_2 = \frac{(qm-2q+rm)}{2}, \quad {\rm and}\quad l_3 = \frac{(qm-q+rm)}{2}
$$
and finally, injecting these values in the third equality above, one finds: 
$$
r=(1-4m)q.
$$
Finally we obtain
\begin{eqnarray}
\nonumber
l_1 & = & -( 2m^2-2m+1)q, \\
\nonumber
l_2 & = & -( 2m^2-m+1)q,  \\
\nonumber
l_3 & = & -(2m^2-m+\frac12 )q. 
\end{eqnarray}
This final equality implies that $q$ must be non-positive and even, say $q=-2q'$ and we finally obtain
\begin{eqnarray}
\nonumber
k_2 & = & (4m^3 -2m^2+2m) q',\\
\nonumber
k_3 & = &   (4m^3 -2m^2+m) q', \\
\nonumber
k_4 & = & (4m^3 -4m^2+2m) q'.
\end{eqnarray}
Computing $k_1$, for instance with \eqref{eqx1}, finally gives $k_1 =  (4m^3 -8m^2+5m-2) q'$.
In view of the constraints \eqref{xis}, this gives $q'=1$ and finally the minimality of $S$.

\medskip

In the case of even $m$, we define the sequence $T(m)$ consisting of
$\begin{pmatrix}
- m \\
- m \\
- m \\
\end{pmatrix}
$
repeated $4m^3-6m^2+4m-1$ times, 
$\begin{pmatrix}
- m \\
m -1\\
m \\
\end{pmatrix}
$
repeated $4m^3-4m^2+2m$ times, 
$\begin{pmatrix}
m\\
-(m-1) \\
m-1\\ 
\end{pmatrix}$   
repeated $4m^3-2m^2+m$ times and finally  
$
\begin{pmatrix}
m-1 \\
m\\
-m 
\end{pmatrix} 
$
repeated $4m^3-4m^2+m$  times. 
We may compute that
\begin{eqnarray*}
\|T(m)\| & = & (4m^3-6m^2+4m-1)+(4m^3-4m^2+2m)+(4m^3-2m^2+m)+(4m^3-4m^2+m) \\
	&=& 16 m^3 -16 m^2 +8m-1.
\end{eqnarray*}

The proof that this is a minimal zero-sum sequence is analogous to the one of the odd case and left to the reader. 

\end{proof}

Finally, we prove Theorem \ref{thmbox3bis}. 

\begin{proof}[Proof of Theorem \ref{thmbox3bis}] Using the notation introduced in Section \ref{newresults}, we observe that
$$
\llbracket -m,m \rrbracket^3  \subseteq \{ (u_1,u_2, u_3) \in \ZZ^3 : u_1^2+ u_2^2 +u_3^2 \leq 3m^2 \} \subseteq \Bcal^{(3)}_{\lceil \sqrt{ 3} m \rceil}.
$$
It now follows from Theorem 5 in \cite{GP26} that
$$\DD^{(4)}  ( \llbracket -m,m \rrbracket ^3) \leq \DD^{(4)} ( \Bcal^{(3)}_{ \lceil\sqrt{ 3} m \rceil}) \sim  \frac{16}{3 \sqrt 3}  \lceil\sqrt{ 3} m \rceil^3 \sim 16m^3.$$
The fact that $\DD^{(4)} ( \llbracket -m,m \rrbracket ^3) \gtrsim 16m^3$, and thus the desired result, follows from Theorem \ref{thmbox3}.

\end{proof}

\end{document}